\documentclass[11pt,a4paper]{amsart}

\usepackage[utf8]{inputenc}
\usepackage[english]{babel}
\usepackage{amsmath}
\usepackage{amssymb}
\usepackage{amsthm}

\usepackage{amsaddr}

\usepackage{marvosym}
\usepackage{multicol}
\usepackage{blindtext}
\usepackage[linesnumbered,ruled,vlined]{algorithm2e}
\usepackage{multicol}
\usepackage{graphicx}
\usepackage{wrapfig}
\usepackage{csquotes}
\usepackage[dvipsnames]{xcolor}
\usepackage{geometry}
\newtheorem{theorem}{Theorem}[section]
\newtheorem*{theorem*}{Theorem}
\newtheorem{lemma}{Lemma}[section]
\newtheorem{definition}{Definition}[section]
\newtheorem*{definition*}{Definition}
\newtheorem{remark}{Remark}[section]

\usepackage{float}
\usepackage{systeme}
\usepackage{hyperref}
\usepackage{bm}

\usepackage{caption}
\usepackage{subcaption}

\def\br{{\bm r}}

\def\bk{{\bm k}}

\def\bbm{{\bm m}}
\def\bl{{\bm l}}
\def\bn{{\bm n}}
\def\AA{{\bm A}}

\usepackage[
    backend=biber,
    style=numeric-comp,
    sorting=ynt
]{biblatex}

\title[Optimal Evaluation of Symmetry-Adapted $n$-Correlations]{Optimal Evaluation of Symmetry-Adapted $n$-Correlations Via Recursive Contraction of Sparse Symmetric Tensors}

\author{Illia Kaliuzhnyi}
    \address{V.N.Karazin Kharkiv National University\\ 4 Svobody Square, Kharkiv, Ukraine}
    \email{ikaliuzh@gmail.com}
\author{Christoph Ortner}
    \address{University of British Columbia\\ 1984 Mathematics Road, Vancouver, BC, Canada V6T 1Z2}
    \email{ortner@math.ubc.ca}

\begin{document}
    \maketitle

    \begin{abstract}
        We present a comprehensive analysis of an algorithm for evaluating 
        high-dimensional polynomials that are invariant under permutations and rotations. The key bottleneck is the
        contraction of a high-dimensional symmetric and sparse tensor with 
        a specific sparsity pattern that is directly related to the symmetries 
        imposed on the polynomial. 
        We propose an explicit construction of a recursive evaluation strategy 
        and show that it is optimal in the limit of infinite polynomial degree.
    \end{abstract}
    
    
    \section{Introduction}
        
Throughout scientific machine learning, explicit enforcement of physical symmetries plays an important role. On the one hand, they are often a {\em requirement} in order to reproduce qualitatively correct physics (e.g., conservation laws). On the other hand, correctly exploiting symmetries can lead to a significant reduction in the number of free parameters and the computational cost of a model. 

The present work is concerned with the efficient evaluation of multi-set functions, 
\[
    p\big( [ \br_j ]_j \big), \qquad \br_j \in \mathbb{R}^d,
\]
where $[ \cdot ]$ denotes a multi-set (mset), 
which are invariant under permutations (implicit in the fact they are mset functions) and equivariant under the action of a symmetry group, typically, $O(d)$ or $SO(d)$. Such symmetries arise, e.g., when modelling properties of atomic environments such as energies, forces, charges, magnetic moments, and so forth. 
A vast variety  of closely related machine-learning frameworks exist to model such properties of particle systems~\cite{Behler2007-ng,Bartok2010-mv,Zuo2020-ba,Musil2021-nv}. Here, we are particularly interested in their representation in terms of {\em symmetry-adapted $n$-correlations} or, equivalently, symmetric polynomials. This approach was pioneered in this application domain by the PIP method \cite{Braams2009-bj}, and the moment tensor potentials \cite{Shapeev2016-pd}. 
The atomic cluster expansion (ACE) \cite{drautz-ace, 2019-ship1, 2020-pace1} and its variants \cite{Seko2019-fe,Nigam2020-pz} provide a general systematic framework for parametrizing structure property relationships of the kind described above. For the sake of simplicity of presentation we will focus only on {\em invariant} parameterizations, but this is the most stringent case and our results therefore apply also to equivariant tensors. 

Thus, we consider the efficient evaluation of (polynomial) mset functions $\varphi$ satisfying 
\[
    \varphi\big([ Q \br_j]_j \big) 
        =
    \varphi\big( [\br_j]_j \big) 
    \qquad \forall Q \in {\rm SO}(d).
\]
The most important case for practical applications is that of three-dimensional particles, i.e., $d = 3$. 
Within the ACE framework, the evaluation of $\varphi$ is eventually (we give details in \S~\ref{sec:ace:background}) reduced to an expression of the form, 
\begin{equation} \label{eq:intro:contraction}
    \varphi = \sum_{\nu = 1}^{\nu_{\rm max}}
        \sum_{\bk} c_{\bk} 
        \prod_{t = 1}^\nu A_{k_t}
\end{equation}
where the $A_{k_t}$ contain features about the particle structures and $\prod_{t = 1}^\nu A_{k_t}$ represent a $\nu$-correlation of the density $\rho = \sum_j \delta(\br - \br_j)$ and can be understood as a {\em basis} in which the property $\varphi$ is expanded with coefficients $c_\bk$. While our own interest originates in modelling atomic properties, our algorithms and results are more generally of interest whenever a large set of $\nu$-correlations is computed that has a structured sparsity induced by a symmetry group. 

For the purpose of implementating \eqref{eq:intro:contraction} it is more natural to interpret ${\bm c} = (c_\bk)_{\bk}$ as a high-dimensional tensor which is contracted against the one-dimensional tensor $(A_k)_k$ in each coordinate direction. 
Sparsity in the tensor ${\bm c}$ arises in three distinct ways: (1) due to permutation-symmetry only ordered tuples $\bk = (k_1, \dots, k_\nu)$ must be considered; (2) a sparse basis is chosen to enable efficient high-dimensional approximation; and (3) the tensor $c_{\bk}$ has additional zeros which arise due to the $O(d)$-invariance and therefore does not represent a downset on the lattice of $\bk$ multi-indices. 

The fact that the set of non-zero indices $\bk$ is not a downset is particularly significant.
Although the naive evaluation cost of \eqref{eq:intro:contraction} only scales linearly with $\nu_{\rm max}$, it still becomes prohibitive for high correlation order $\nu_{\rm max}$. 
In \cite{2020-pace1, 2019-ship1} a heuristic was proposed to evaluate \eqref{eq:intro:contraction} recursively which appeared to significantly lower the computational cost on a limited range of model problems. However, the recursion strategy is not unique, and no indication was given whether the specific choice made in \cite{2020-pace1} is guaranteed to improve the computational cost, let alone be close to optimal. 
In the present work we will fill these gaps by showing that the strategy of \cite{2020-pace1, 2019-ship1} is indeed quasi-optimal, as well as presenting a new explicit rule to construct the recursion which is asymptotically optimal in the limit of large polynomial degree.

\section{The ACE Model}
\subsection{Background: Many-body ACE Expansion}
\label{sec:ace:background}
A (local) configuration of identical non-colliding particles is described by a set $[ \br_j ]_{j = 1}^J \subset \mathbb{R}^d$. The qualifier {\em local} is used to indicate that the positions $\br_j$ are normally relative positions with respect to some central particle. We are concerned with the parametrization of invariant properties of such configurations, i.e., mappings $\varphi$ that are invariant under a group action, 
\begin{equation*}
    \varphi\big( [Q \br_j]_j \big) = 
    \varphi\big( [\br_j]_j \big)  \qquad \forall Q \in G,
\end{equation*}
where $G$ is an orthgogonal group, in our case we will consider $G = O(3)$, $G = SO(2)$

A natural approach to parametrize such properties is the many-body expansion (or, high-dimensional model reduction, HDMR), where $\varphi$ is approximated by
\[
    \sum_{\nu = 1}^{\nu_{\rm max}} 
    \sum_{j_1 < \dots < j_\nu} \varphi^{(\nu)}(\br_{j_1}, \dots, \br_{j_\nu}). 
\]
The computational cost of the inner summation over all $\nu$-clusters scales combinatorially, and quickly becomes prohibitive. In particular, the majority of models we are aware of truncate the expansion at $\nu_{\rm max} = 3$. 

The atomic cluster expansion (ACE) \cite{drautz-ace} can be thought of as a mechanism to represent such a many-body expansion in a computationally efficient way. Here, we give only an outline and refer to \cite{drautz-ace,2019-ship1,Musil2021-nv} for further details. 
Briefly, the idea is to replace the summation $\sum_{j_1 < \dots < j_\nu}$ over a discrete simplex with a summation over a tensor product set, 
\[
    \varphi([\br_j]_j) = 
    \sum_{\nu = 1}^{\nu_{\rm max}}
    \sum_{j_1, \dots, j_\nu} 
    u^{(\nu)}(\br_{j_1}, \dots, \br_{j_\nu}). 
\]
Next, we parametrize $u^{(\nu)}$ by a tensor product basis, 
\[
    u^{(\nu)}(r_{j_1}, \dots, r_{j_\nu}) 
    = \sum_{k_1, \dots, k_\nu} c_\bk\prod_{t = 1}^\nu \phi_{k_t}(\br_{j_t}),
\]
where $\phi_k(r)$, indexed by a symbol $k$ that could represent a multi-index, is called the {\em one-particle basis}, and we will say more about how this basis is chosen below. 
Inserting this expansion and reordering the summation we arrive at 
\begin{align*}
    \varphi\big([\br_j]_{j=1}^J\big) 
    = 
    \sum_{\nu = 1}^{\nu_{\rm max}} 
    \sum_{k_1, \dots, k_\nu}
    c_\bk 
    \sum_{j_1, \dots, j_\nu} \prod_{t = 1}^\nu \phi_{k_t}(r_{j_t}) 
    \\ 
    = 
    \sum_{\nu = 1}^{\nu_{\rm max}} 
    \sum_{k_1, \dots, k_\nu} 
    c_\bk 
    \prod_{t = 1}^\nu
    \sum_{j = 1}^J \phi_{k_t}(\br_j).
\end{align*}

Absorbing the sum over the {\em correlation order} $\nu$ into $\sum_{k_1, \dots}$ we obtain the parametrization 
\begin{equation} \label{eq:ACE_param}
    \begin{split}
    \varphi\big([\br_j]_{j=1}^J\big) 
    &= 
    \sum_{\bk \in \mathcal{K}} c_\bk \AA_\bk \\ 
    \AA_\bk &:= \prod_{t = 1}^{\nu(\bk)} A_{k_t}, \\ 
    A_k &:= \sum_{j = 1}^J \phi_k(\br_j),
    \end{split}
\end{equation}
where $\mathcal{K}$ is a set of tuples specifying which basis functions $\AA_\bk$ are used in the parametrization, and $\nu(\bk)$ is the correlation order of that basis function (i.e. the length of the tuple $\bk$).
It is shown rigorously in \cite{2019-ship1, BachDussOrt2021} under natural assumptions on the one-particle basis, that in the limit of infinite correlation order and infinite basis size $\mathcal{K}$ this expansion can represent an arbitrary regular set function $\varphi$.

\subsection{\texorpdfstring{$SO(2)$-}{}Invariance}
\label{sec:SO2}
If the particle system is two-dimensional, $d = 2$, then we describe particle positions in radial coordinates, $\br = r (\cos\theta, \sin\theta)$ and choose as single particle basis functions 
\[
    \phi_k(\br) \equiv \phi_{mn}(\br) := R_n(r)e^{-i\theta m}, 
\]
where we have identified $k \equiv (m, n)$, introduced a radial basis $R_n$ and used trigonometric polynomials to discretise the angular component. 
It is then straightforward to see that 
\[
    {\int\hspace{-1.3em}-\!\!-}
        \AA_{\bbm \bn}\big( \big[ e^{i \theta} \br_j \big]_j \big) \, d\theta 
    = 
    \begin{cases} 
        \AA\big( \big[ \br_j \big]_j \big), & \sum_t m_t = 0, \\ 
        0, & \text{otherwise}.
    \end{cases}
\]
This implies that the subset of basis functions $\AA_{\bbm \bn}$ for which $\sum_t m_t = 0$ constitute a rotation-invariant basis. Thus we define the set of all {\em invariant} basis functions, represented by their multi-indices, 
\[
    \mathcal{K}^{inv}_{O(2)} := \big\{
        \bk = [(n^t, m^t)]_{t = 1}^\nu
        \,|\,
        k^t \in \mathbb{Z}_+ \times \mathbb{Z}, \,
        \nu \in \mathbb{N},\,
        {\textstyle \sum_{t = 1}^\nu m^t = 0} 
    \big\}.
\]
Full $O(2)$-invariance (reflections) can be obtained by simply taking the real part of the basis, i.e., replacing $\AA_\bk$ with ${\rm real}(\AA_\bk)$, hence we ignore this additional step and focus on $SO(2)$ invariance.

\begin{remark}
    $O(2)$ and $SO(2)$ invariance naturally occurs also in the three-dimensional setting when considering a cylindrical coordinate system, e.g., when modelling properties of bonds. In that case one would obtain a product symmetry group $O(2) \otimes O(1)$ where $O(1)$ is associated with the reflection in the $z$-coordinate.
\end{remark}


\subsection{The one-dimensional torus: \texorpdfstring{$\mathbb{T}$}{T}}
\label{sec:TT}
The simplest non-trivial case we consider is to let all the particles lie on the unit circle. 
The particle positions are now described simply by their angular component $\br_j \equiv e^{i \theta_j}$, and we can ignore the radial component $r_j$ and hence the radial basis $R_n$. The one-particle basis is then given simply by $\phi_m(e^{i\theta}) := e^{i \theta m}$. 
The reason this case is of particular interest is that the main challenge in the construction and analysis of the recursive evaluator occur due to the $m$ components. With a tensor product decomposition for $\phi_{nm}$ it will be straightforward to extend our results to that case. The symmetry group $SO(2)$ can now be identified with the torus itself, hence we denote it by $\mathbb{T}$. 
Here the set of invariant basis functions is represented by the tuples 
\[
    \mathcal{K}^{\rm inv}_{\mathbb{T}} := 
    \big\{ 
        \bbm = [m^t]_{t=1}^{\nu} \,\big|\, m^t \in \mathbb{Z},\, \nu \in \mathbb{N}, 
                {\textstyle \sum_{t = 1}^\nu m^t = 0} 
    \big\}.
\]

\subsection{\texorpdfstring{$O(3)$}{O(3)}-Invariance}
\label{sec:O3}
To incorporate rotation-invariance into the parametrization \eqref{eq:ACE_param} when $d = 3$ we identify $k \equiv (nlm)$, and choose  the one-particle basis 
\[
    \phi_k(\br) := \phi_{nlm}(\br) := R_n(r) Y_l^m(\hat{\br}), 
\]
where $Y_l^m$ are the standard complex spherical harmonics. 
By exploiting standard properties of the spherical harmonics it is possible to enforce rotation-invariance as an explicit constraint on the parameters $c_\bk$ in \eqref{eq:ACE_param}. In practice one actually first constructs a second rotation-invariant basis and then converts the parametrization to \eqref{eq:ACE_param} for faster evaluation~\cite{2019-ship1,2020-pace1}.

The details are unimportant for our purposes, except for one fact: a parameter $c_\bk$ can only be non-zero if it belongs to the set
\[
    \mathcal{K}^{\rm inv}_{O(3)} := 
    \Big\{ 
        \bk = [(n^t, l^t, m^t)]_{t = 1}^\nu 
        \,\Big|\, 
         k^t \in \mathbb{Z}_+ \times \mathbb{Z}_+ \times \mathbb{Z}, \,
        \nu \in \mathbb{N}, \,
        \sum_t m_t = 0, \,
         \sum_t l_t \text{ is even}
     \Big\}.
\]
Thus, we arrive at a very similar structure for the set $\mathcal{K}^{\rm inv}_{O(3)}$ as in the $SO(2)$-invariant case. The additional constraint that $\sum_t l_t$ must be even, arises from inversion symmetry. 

\subsection{Sparse polynomials}
\label{sec:sparse_polys}
In all of the three representative cases, $\mathbb{T}$, $SO(2)$, $O(3)$, we arrived at the situation that the representation \eqref{eq:ACE_param} only involves basis functions $\AA_\bk$ from a strict subset $\mathcal{K}^{\rm inv}$ of all possible tuples $\bk$. For practical implementations we must further reduce this infinite index set to a finite set, thus specifying a concrete finite representation. 

When the target function $\varphi$ we are trying to approximate through our parameterisation is analytic, in the sense that the $\nu$-body components $u^{(\nu)}$ are analytic, then approximation theory results \cite{BachDussOrt2021} suggest that we should use a {\em total degree sparse grid} (or, simply, sparse grid) of basis functions. This will be the main focus of our analysis, but for numerical tests we consider the more general class 
\begin{equation}  \label{eq:sparsegrid}
    \mathcal{K}_G^p(D) := \big\{ \bk \in \mathcal{K}^{\rm inv}_G : \| \bk \|_p \leq D \big\}, 
\end{equation}
where $D > 0$ is the {\em degree} and $p > 0$ a parameter that specifies how the degree of a basis function $\AA_\bk$, is calculated. In the three cases we introduced above, the degree, $\|\bk\|_p$, is defined as 
\[
    \begin{cases}
    \| \bk \|_p
    = 
    \| ({\bf nlm}) \|_p
    = 
    \Big(\sum_t \big( n_t + l_t \big)^p\Big)^\frac{1}{p}, & \text{ case } O(3), \\ 
    \| \bk \|_p
    = 
    \| ({\bf nm}) \|_p
    = 
    \Big(\sum_t \big( n_t + |m_t| \big)^p\Big)^\frac{1}{p}, & \text{ case } SO(2), \\ 
    \| \bk \|_p
    = 
    \| {\bf m} \|_p
    = 
    \Big( \sum_t m_t^p \Big)^\frac{1}{p}, & \text{ case } \mathbb{T}, 
    \end{cases}
\] 
Note that $|m_t| \leq l_t$ hence it does not appear in this definition for $d = 3$.
The total degree is obtained for $p = 1$. 
In the context of the ACE model, this choice has been proposed and used with considerable success in \cite{2019-ship1,2020-pace1, 2021-acetb1}.

\subsection{Recursive Evaluation of the Density Correlations}
Aside from the choice of radial basis $R_n$ (which is not essential to the present work) we have now fully specified the basis for the ACE parameterisation \eqref{eq:ACE_param}. Typical basis sizes range from 1,000 to 100,000 in common regression tasks. It has been observed in \cite{2019-ship1,2020-pace1} that at least for larger models the product basis evaluation as a products of one-particle basis functions is the computational bottleneck. Our task now is to evaluate the basis (and hence the model $\varphi$) as efficiently as possible. Towards that end, a recursive scheme was proposed in \cite{2019-ship1,2020-pace1}, in which high correlation order functions are computed as a product of two lower order ones.


Consider a basis (multi-) index $\bk := (\bn \bl \bbm)$, (e.g. $k^t = (n^t, l^t, m^t)$, if $G = O(3)$. We say that  $\bk = [k^t]_{t = 1}^{\nu_1 + \nu_2}$ has a decomposition $\bk = [\bk_1, \bk_2]$ where $\bk_i = [k_i^t]_{t = 1}^{\nu_i}$ if $\bk$ is the mset-union of $\bk_1, \bk_2$, i.e., 
\[
    [k^1, \dots, k^\nu] = 
    [k_1^1, \dots k_1^{\nu_1}, k_2^1, \dots, k_2^{\nu_2}],     
\]
This is equivalent to the decomposition of the basis function,   
\begin{equation}
    \AA_{\bk} = \AA_{\bk_1} \cdot \AA_{\bk_2}.
\end{equation}
$\AA_{\bk}$ can then be computed with a single product, provided of course that $\bk_1, \bk_2 \in\mathcal{K}$ as well.


The crux is that $\bk$ as well as $\bk_1, \bk_2$ must satisfy the mentioned symmetry constraints reviewed in \S\ref{sec:SO2}, \S\ref{sec:TT}, \S\ref{sec:O3}; most notably $\sum_t m_t = 0$.  For the purpose of illustration consider only an $m$-channel, i.e., the torus case $G = \mathbb{T}$. For example, $\bk = [1, 0, -1]$ can clearly be decomposed into $\bbm_1 = [1, -1]$ and $\bbm_2 = [0]$ and, therefore, $\AA_{1, 0, -1} = A_{0} \cdot \AA_{1, -1}$. But some other basis functions do not have a {\em proper} decomposition: for instance, the tuple $\bbm = [1, 1, -2]$ cannot be decomposed. 
We will call such functions {\em independent}, due to the fact that these are precisely the algebraically independent basis functions that cannot be written as polynomials of lower-correlation order terms. 

\begin{definition}
    Let $\bk \in \mathcal{K}^{\rm inv}$, then we call $\bk$ {\em dependent}, if there exist $\bk_1, \bk_2 \in \mathcal{K}^{\rm inv}$ such that $\bk = [\bk_1, \bk_2]$. Otherwise we say that $\bk$ is {\em independent.}
\end{definition}

To decompose independent functions we add auxiliary basis functions that are not invariant under the above mentioned symmetries and therefore do not occur in the expansion \eqref{eq:ACE_param}. For example, $\bbm = [1, 1, -2]$ could be decomposed into $[1,1]$ and $[-2]$ where the latter is simply an element of the atomic base $A_{-2}$. That is, only a single auxiliary basis function $\AA_{1,1}$ is required. Auxiliary basis functions $\AA_\bk$ are simply assigned a zero parameter $c_\bk$ in the expansion \eqref{eq:ACE_param}.

These ideas result in a directed acyclic computational graph 
\[
    \mathcal{G} = \{ \bk \equiv [\bk_1, \bk_2] \}
\]
with each node $\bk$ representing a basis function. The graph determines in which order basis functions need to be evaluated to comply with the recursive scheme. So,
\begin{equation}
    e: \bk_1 \to \bk \in \text{Edges}(\mathcal{G}) \implies \exists \bk_2 \in \text{Vertices}(\mathcal{G}):  A_{ \bk} = A_{ \bk_1} \cdot A_{\bk_2}  
\end{equation}
It means that every node of correlation order $\geq 2$ has exactly two incoming edges and possibly zero or more outcoming ones. 
To construct $\mathcal{G}$ we first insert the nodes $\{\bk = [k]\}$ corresponding to the 1-correlation basis functions $\{A_{k}\}$ and further nodes in increasing order of correlation. The crucial challenge then is to find an insertion algorithm that minimizes the number of auxiliary nodes inserted into the graph and thus optimizes the overall computational cost. 
        
    \section{Summary of Main Results}
        Our construction and analysis of node insertion algorithms naturally relies on traversing by increasing correlation order. Therefore, we begin by defining 
\[
    \mathcal{K}_G(\nu, D) := \big\{ \bk \in \mathcal{K}^{\rm inv}_G \, \big|\, 
                                    \| \bk \|_1 \leq D, \nu(\bk) = \nu \big\},
\]
for each of the groups $G = \mathbb{T}, SO(2), O(3)$. Written out concretely, these sets are given by 
\begin{align*}
    \mathcal{K}_{\mathbb{T}}(\nu, D) &:= 
    \Big\{ \bk = \bbm \,\big|\, m^t  \in \mathbb{Z}, \, {\textstyle \sum_{t = 1}^{\nu} m^t = 0}, \, 
                {\textstyle \sum_{t = 1}^\nu |m^t| \leq D} \Big\}, \\
    \mathcal{K}_{SO(2)}(\nu, D) &:= 
    \Big\{ \bk = [ (n^t, m^t)]_{t = 1}^\nu \,\big|\,
    {\textstyle k^t \in \mathbb{Z}_+ \times \mathbb{Z}}, \,
    {\textstyle \sum_{t = 1}^{\nu} m^t = 0}, 
    {\textstyle \sum_{t = 1}^\nu n^t + |m^t| \leq D} \, \Big\}, \\
    \mathcal{K}_{O(3)}(\nu, D) &:= \Big\{  \bk = \big[(m^t, l^t, n^t)\big]_{t = 1}^{\nu} \,\Big|\,
    {\textstyle k^t \in \mathbb{Z}_+ \times \mathbb{Z}_+ \times \mathbb{Z}}, \,
    {\textstyle \sum_{t = 1}^{\nu} m^t = 0}, \quad {\textstyle \sum_{t=1}^{\nu}l^t} \text{ is even}, \nonumber\\
     & \hspace{7.7cm} {\textstyle |m^t| \leq l^t, \, \sum_{t = 1}^\nu n^t + l^t \leq D} \Big\}. \nonumber
\end{align*}
\subsection{Most Basis Functions Are Independent}
The subsets of {\em dependent} and {\em independent} basis functions are, respectively, defined by 
\begin{equation} \nonumber
    \begin{split}
    \mathcal{D}_{G}(\nu, D) 
    &:= 
    \big\{ \bk \in \mathcal{K}_{G}(\nu, D)\, \big| \, \bk \text{ is dependent} \big\}, \\
    \mathcal{I}_G(\nu, D) &:= 
    \big\{ \bk \in \mathcal{K}_{G}(\nu, D)\, \big| \, \bk \text{ is independent} \big\}.
    \end{split}
\end{equation}
We will denote auxiliary basis functions, which are needed for evaluation in compliance with a recursive scheme, of correlation order no more than $\nu$ and degree no more than $D$ by $\mathcal{A}_{G}(\nu, D)$. Note, that $\mathcal{A}_{G}$ depends on a particular insertion algorithm we use. The total number of basis functions (or, vertices in the computational graph) therefore becomes 
\[
    \mathcal{V}_{G}(\bar\nu, D) := \mathcal{A}_{G}(\bar\nu, D) \cup \bigcup_{\nu=1}^{\bar\nu} \mathcal{K}_{G}(\nu, D).
\]
We will prove in \S~\ref{sec:proofs} the intuitive statements that, for $\nu < \bar\nu$,  $\# \mathcal{K}_{G}(\nu, D) = \overline{o}\big(\# \mathcal{K}_{G}(\bar\nu, D)\big)$ as $D \to \infty$, and $\# \mathcal{D}_{G}(\nu, D) = \overline{o}\big(\# \mathcal{D}_{G}(\bar\nu, D)\big)$. That is, asymptotically, nodes of the highest correlation order constitute the prevailing majority of nodes in the computational graph. In particular this means that 
\[
    \#\mathcal{V}_{G}(\nu, D) = \Theta\big(\#\mathcal{K}_{G}(\nu, D) + \#\mathcal{A}_{G}(\nu, D)\big), \qquad \text{as } D \to \infty.
\]
After these preparations, we obtain the following result which states that asymptotically the vast majority of nodes are independent. This result is numerically confirmed in Figure \ref{fig:depT} for the $\mathbb{T}$  case.

\begin{theorem}[Independent nodes prevalence]
    Let $\nu \geq 3$ and $G \in \{\mathbb{T}, SO(2), O(3)\}$, then 
    \label{results:th:3dasymp}
    \[
        \frac{\text{\#}\mathcal{D}_{G}(\nu, D)}{\text{\#}\mathcal{K}_{G}(\nu, D)} = \Theta\left(\frac{1}{D}\right), \qquad \text{as } D \to \infty.
    \]
\end{theorem}
This result highlights how important it is to design the insertion of auxiliary nodes in the recursive evaluation algorithm with great care so as not to add too many auxiliary nodes. 

\subsection{Original Insertion Heuristic}
Algorithm~\ref{alg:original} below is a formal and detailed specification of the heuristic proposed in \cite{2020-pace1}. The key step is line 11: if a new basis function $\bk = [k^1, \dots, k^t]$ cannot be split into $\bk_1, \bk_2$ that are already present in the graph, then we simply ``split off'' the highest-degree one-particle basis function. The idea is that the remaining $(\nu-1)$-order basis function will typically have relatively low degree and there are therefore fewer of such basis functions to be added into the graph.

\begin{algorithm} 
    \SetAlgoLined
    \DontPrintSemicolon
    \caption{Recursively insert node $\bk$ into graph $\mathcal{G }$ as proposed in \cite{2020-pace1}, 
    \label{alg:original}}
    
    \SetKwProg{Fn}{Function}{:}{}
    \Fn{Insert($\bk :: \text{node}$)}{
        \If{$\bk \in \mathcal{G}$}{
            \Return
        }
        \For{all decompositions $\bk \equiv [\bk_1, \bk_2]$}{
            \If{$\bk_1, \bk_2 \in \mathcal{G}$}{
                $\mathcal{G} \leftarrow \mathcal{G} \cup \{\bk \equiv [\bk_1, \bk_2]\}$  \\ 
                \Return
            }
        }
        Identify $\bk \equiv [  [k^1] , \bk']$ where $\|k^1\|_p = \max(\|k^1\|_p, \|k^2\|_p, \dots)$ \\
        \textit{Insert}($\bk'$) \tcp*{Recursively insert $\bk'$; $[k^1]$ always belongs to $\mathcal{G}$} 
        \textit{Insert}($\bk$) \tcp*{now $\bk$ can be inserted as well} 
        \Return
    }
\end{algorithm}

It turns out that this fairly naive heuristic is already close to optimal, which is the first main result of this paper.

\begin{theorem}[Complexity of Algorithm \ref{alg:original}] Suppose that $\nu \geq 3$ is held fixed.
    \label{theorem-alg1}
    {\it (i) } If $G = SO(2)$ or $G = O(3)$ then the number of auxiliary nodes inserted by Algorithm \ref{alg:original} behaves asymptotically as 
    \[
        \frac{\#\mathcal{A}_{SO(2)}(\nu, D)}{\#\mathcal{V}_{SO(2)}(\nu, D)} = O\left(\frac{1}{D}\right); \qquad
                \frac{\#\mathcal{A}_{O(3)}(\nu, D)}{\#\mathcal{V}_{O(3)}(\nu, D)} = O\left(\frac{1}{D^2}\right).
    \]

    {\it (ii)} If $G = \mathbb{T}$, then Algorithm \ref{alg:original} inserts exactly 
    \[
        \#\mathcal{A}_{\mathbb{T}}(\nu, D) = 2\sum_{k = 2}^{\nu - 1}\sum_{n = 1}^{\lfloor D/2 \rfloor} \pi(k, n)
    \] 
    auxiliary nodes, where $\pi(k, n)$ is the number of integer partitions of $n$ into exactly $k$ parts. Moreover, 
    \[
                \lim_{D \to \infty} \frac{\#\mathcal{A}_{\mathbb{T}}(\nu, D)}{\#\mathcal{V}_{\mathbb{T}}(\nu, D)} = \frac{1}{1 + C_{\nu - 1} / 2}, \qquad \text{where } \quad C_{\nu -1} = \frac{1}{\nu - 1}\binom{2\nu - 2}{\nu - 2},
    \]
    i.e., $C_{\nu - 1}$ is the $(\nu - 1)$th Catalan number.
\end{theorem}

\begin{remark}
    Note that expanding the Catalan numbers yields 
    $
        \frac{1}{1 + C_{\nu}/2} \sim 2 \sqrt{\pi} \nu^{3/2} 4^{-\nu} \text{as } \nu \to \infty, 
    $
    which suggests that at high correlation order relatively few auxiliary nodes are inserted. 

    However, this comes with a caveat: The double-limit $\lim_{D, \nu \to \infty}$ is likely ill-defined; that is, we expect that the balance between $D$ and $\nu$ as the limit is taken leads to different asymptotic behaviour.
\end{remark}

For the $SO(2)$ and $O(3)$ cases the foregoing theorem establishes that ``very few'' auxiliary nodes are required, at least at high polynomials degree. However, the case $G = \mathbb{T}$ highlights that there is space for further improvement. While we still see that relatively few auxiliary nodes are required at high degree {\em and} high correlation order, this is clearly not true in the pre-asymptotic regime. This would also be important in the $G = SO(2), O(3)$ cases if the balance between radial and angular basis functions is chosen different, i.e., if a relatively small radial basis were used. This motivates us to explore alternative algorithms. 
%
%
%
%
%
%
\begin{algorithm} 
    \SetAlgoLined
    \DontPrintSemicolon
    \caption{Generalized Insertion Heuristic
    \label{alg:modn}}
    
    \SetKwProg{Fn}{Function}{:}{}
    \Fn{Insert($\bk :: \text{node}, n :: \mathbb{N}$)}{
        \If{$\bk \in \mathcal{G}$}{
            \Return
        }
        
        \For{all decompositions $\bk \equiv [\bk_1, \bk_2]$}{
            \If{$\bk_1, \bk_2 \in \mathcal{G}$}{
                $\mathcal{G} \leftarrow \mathcal{G} \cup \{\bk \equiv [\bk_1, \bk_2]\}$  \\ 
                \Return
            }
        }
        
        Identify $\bk \equiv \large[[k^1, \dotsc, k^s]  , \bk'\large]$ where $s\leq \min(n,\, len(\bk) - 1)$ and $\|[k^1, \dotsc k^s]\|$ is maximal \\
        \textit{Insert}($[k^1, \dotsc, k^s]$, $1$) \tcp*{This node is inserted with Algorithm \ref{alg:original}}
        \textit{Insert}($\bk'$, $n$) \\ 
        \textit{Insert}($\bk$, $n$) \\
        \Return
    }
\end{algorithm}
\subsection{Generalized Insertion Heuristic}
Algorithm~\ref{alg:modn} is a generalization of Algorithm~\ref{alg:original}, allowing a cutoff of multielement subtuples with maximal degree of length no more than a parameter $n$. In this case the original scheme is the special case $n = 1$. Our main interest in this algorithm is that we can establish significantly improved asymptotic behaviour.
\begin{theorem}[Complexity of Algorithm \ref{alg:modn}]
    \label{theorem-alg2}
    The number of auxiliary nodes inserted by Algorithm \ref{alg:modn} with parameter $n \leq \nu/2$ scales as 
    \[
        \frac{\#\mathcal{A}_{G}(\nu, D)}{\#\mathcal{V}_{G}(\nu, D)} =
        \begin{cases}
            O(D^{1-n}), & G = \mathbb{T}, \\ 
            O(D^{1-2n}), & G = SO(2), \\ 
            O(D^{1 - 3n}), & G = O(3). \\ 
        \end{cases}
    \]
\end{theorem}
\subsection{Computational tests}
\label{sec:experimental}
\begin{figure}
    \centering
    \includegraphics[width = 0.5\textwidth]{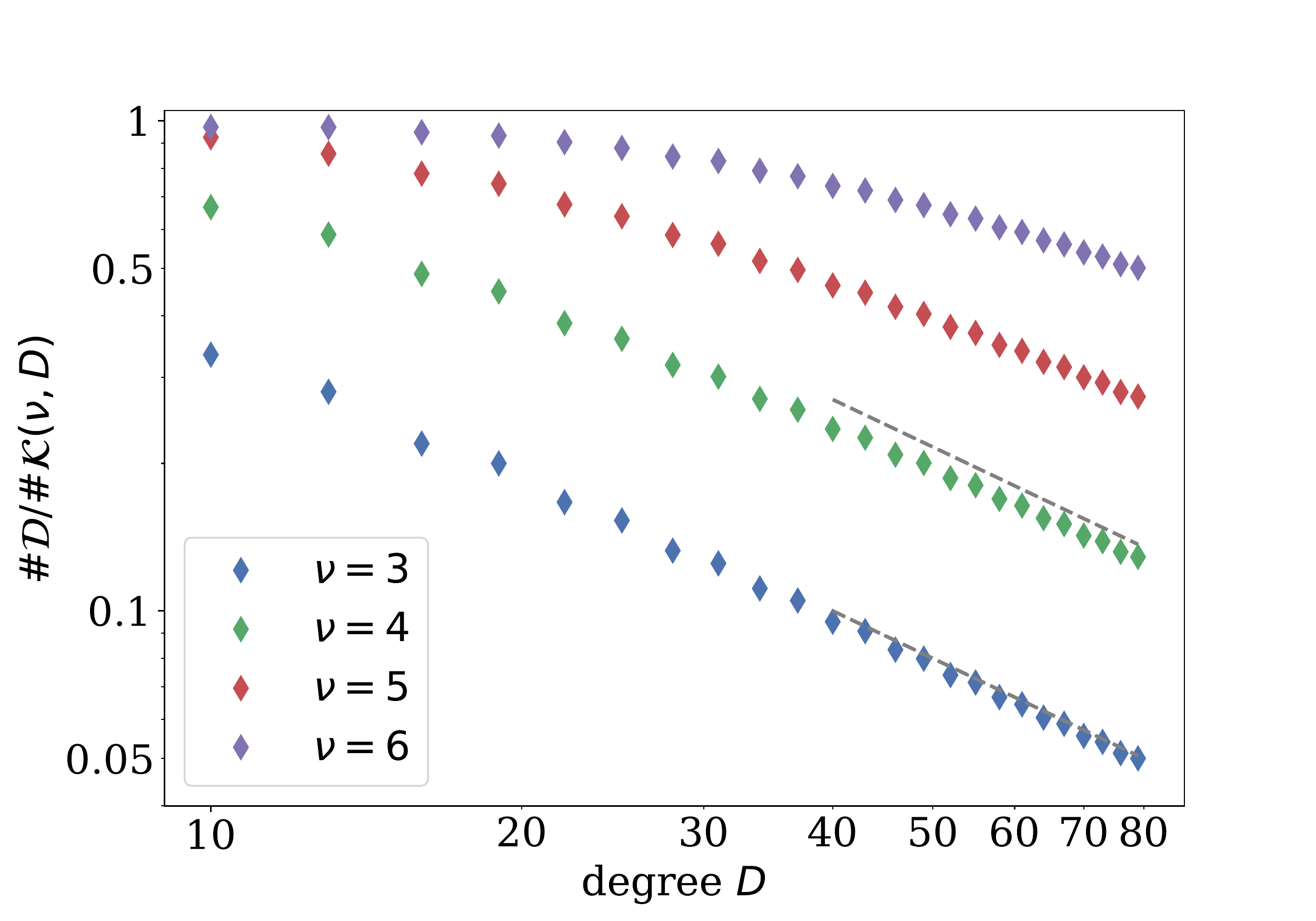}
    \caption{Ratio of dependent to total number of basis functions for the torus case, $G = \mathbb{T}$; dashed lines indicated $\sim 1/D$. }
    \label{fig:depT}
\end{figure}
\begin{figure}
    \centering
    \includegraphics[width = 1\textwidth]{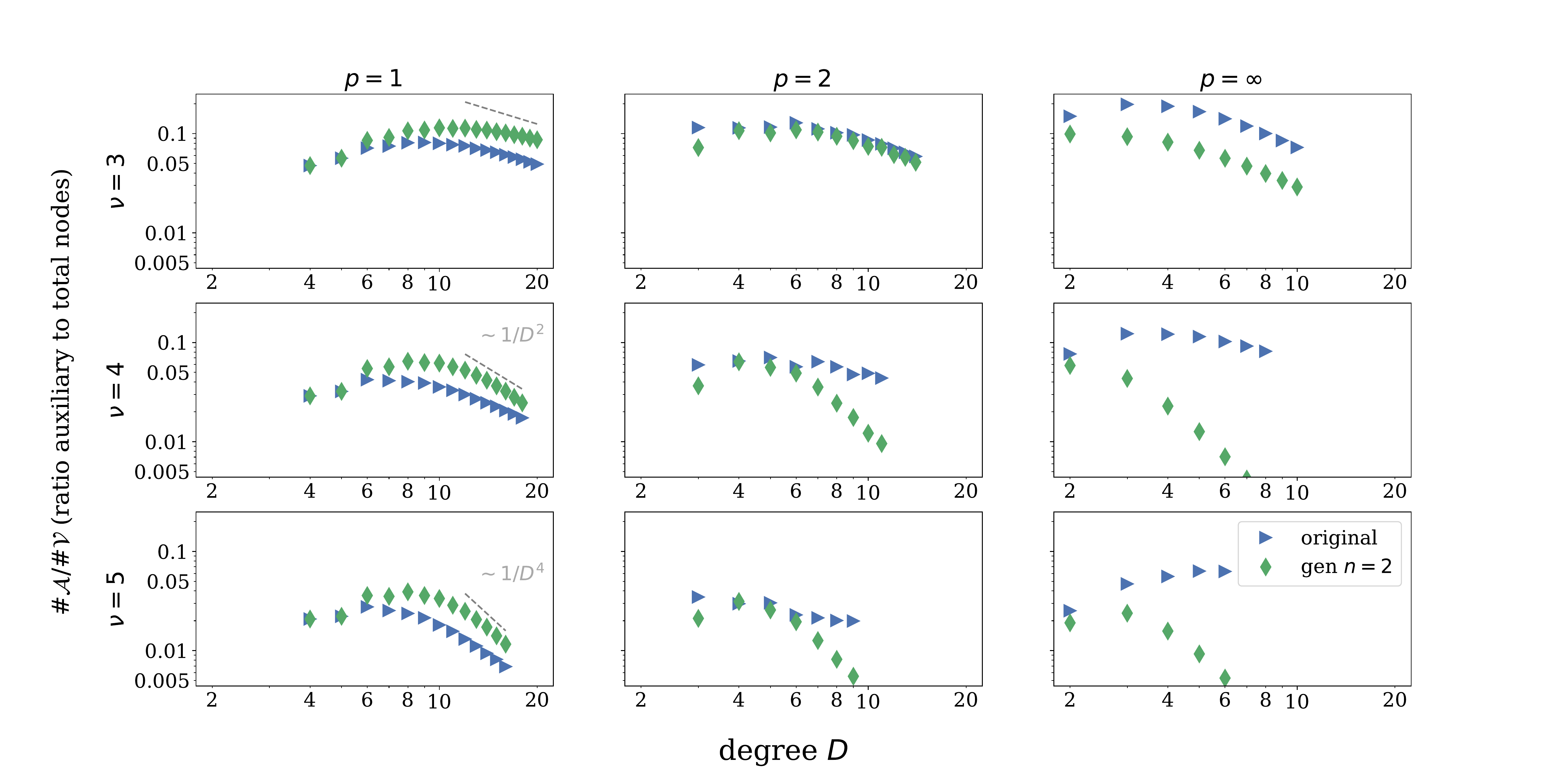}
    \caption{Explicit computation of $\frac{\text{\#}\mathcal{A}_{O(3)}(\nu, D)}{\text{\#}\mathcal{V}_{O(3)}(\nu, D)}$ in the preasymptotic regime;
    Algorithm~\ref{alg:modn} is used with parameter $n = 2$.}
    \label{fig:o3p1}
\end{figure}
\begin{figure}
    \centering
    \includegraphics[width = 0.8\textwidth]{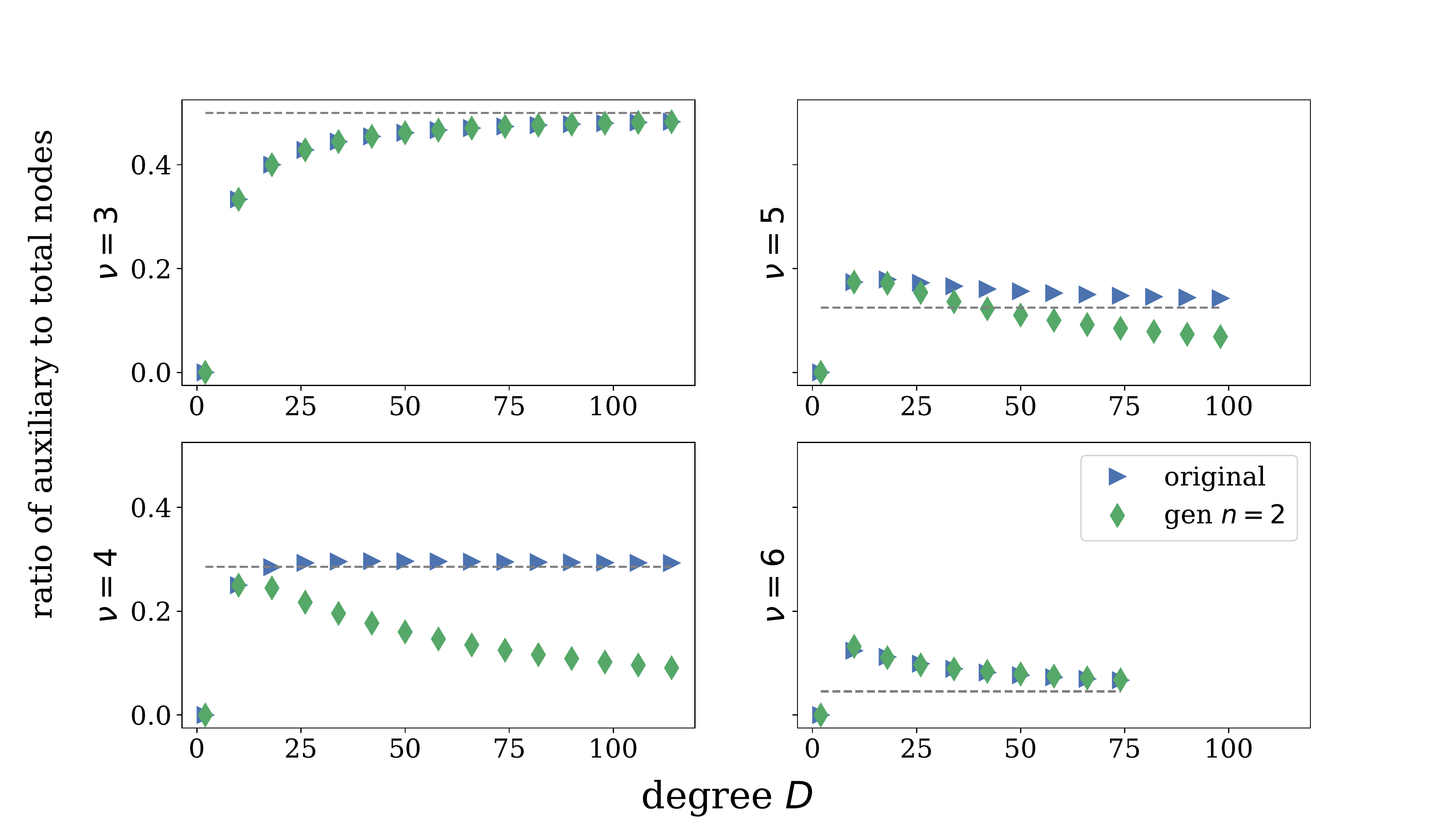}
    \caption{Explicit computation of $\frac{\text{\#}\mathcal{A}_{T}(\nu, D)}{\text{\#}\mathcal{V}_{T}(\nu, D)}$ in case $p = 1$; Algorithm~\ref{alg:modn} is used with parameter $n = 2$. The dotted line indicates the predicted limit of $\frac{\text{\#}\mathcal{A}_{T}(\nu, D)}{\text{\#}\mathcal{V}_{T}(\nu, D)}$ as $D \to \infty$ for Algorithm~\ref{alg:original}; cf. Theorem \ref{theorem-alg1} ($ii$). }
    \label{fig:tp1}
\end{figure}
We performed computational tests to confirm the predictions of Theorems~\ref{theorem-alg1} and~\ref{theorem-alg2}, as well as to expore the performance of our insertion algorithms in the preasymptotic regime, and for different notions of polynomial degree. To this end we generated the computational graphs for the groups $G = \mathbb{T}, O(3)$ for varying correlation order and polynomial degree and plotted the ratios $\frac{\text{\#}\mathcal{A}_{G}(\nu, D)}{\text{\#}\mathcal{V}_{G}(\nu, D)}$ of auxiliary versus total nodes. 

In Figure~\ref{fig:tp1} we show the results for the torus case, $G = \mathbb{T}$. In the case $\nu=3$ every independent node requires a unique auxiliary one to be inserted regardless of the insertion scheme. For $\nu = 4, 5$ we obtain a clear confirmation of our theoretical results. In particular we observe a clear improvement for $\nu = 4, 5$. Our result for $\nu = 6$ is still consistent with our theory, but the improvement of Algorithm~\ref{alg:modn} is no longer visible in the regime that we can easily reach in these tests, likely due to the fact that the asymptotic behaviour of Algorithm~\ref{theorem-alg1} is already very close to optimal. Moreover, Figure \ref{fig:depT} indicates that there are relatively few independent nodes in the pre-asymptotic regime at high correlation orders, which likely plays a role here as well.

In Figure~\ref{fig:o3p1} we compare the two insertion heuristics for the three-dimensional case $G = O(3)$. We show results for the total degree case case ($p = 1$) as well as for less sparse polynomial basis constructions, namely the cases $p = 2, \infty$ described in \S~\ref{sec:sparse_polys}. These are also interesting for applications but more difficult to tackle rigorously. For $p = 1$ we observe that, even though the generalized algorithm with $n > 1$ has far superior asymptotic behaviour than the original heuristic, the preasymptotic behaviour is in fact slightly worse. Specifically, one should use the generalized scheme only in the $D >> \nu$ regime. Note, however, that it is natural to expect from an algorithm with these superior asymptotics in $p=1$ to perform better for $p > 1$ as $\mathcal{K}^{p > 1}(\nu, D)$ contains nodes with $\|k\|_1 > D$. This intuition is also clearly confirmed by our tests shown in Figure~\ref{fig:o3p1}.
\section{Conclusions}
Our analysis provides another compelling argument for the outstanding performance of the atomic cluster expansion method for parameterising symmetric functions of many variables, introduced and further developed in \cite{drautz-ace,BachDussOrt2021,2020-pace1}.

Specifically, we presented a first analysis of an algorithm for generating a computational graph to efficiently evaluate symmetric polynomials in a format reminiscent of power sum polynomials where the ``basis lattice'' has holes that are due to different symmetry constraints. The key step is to understand the insertion of so-called ``auxiliary nodes'' into this graph which represent intermediate computational steps. Our two main results are (1) to explain and establish rigorously that the insertion scheme proposed in \cite{2020-pace1} is already asymptotically optimal in certain regimes (high degree and/or correlation order); and (2) to propose a generalized insertion algorithm with significantly improved asymptotic performance, as well as promising pre-asymptotic performance outside of the total-degree approximation regime. In \ref{sec:invariant_features} we briefly analyze the case when other invariant features (e.g. electric charge or atomic mass) are considered as well. 

An important next step will be to study the optimization of our algorithms for different architectures, in particular for GPUs
\section{Proofs}
    \label{sec:proofs} 

\subsection{Integer Partitions}
    The number of integer partitions of $D \in \mathbb{N}$ into exactly $\nu \in \mathbb{N}$ parts is denoted by $$\pi(\nu, D) = \text{\#}\Big\{\bk \in \mathbb{N}^{\nu} \,\Big|\,  k^1 \geq k^2 \geq \dotsc \geq k^{\nu}, \, \sum_{i = 1}^{\nu}k^i = D\Big\}$$
We further define $\pi(0, D) = 0$. It satisfies the bounds 
\begin{equation}
        \frac{1}{\nu!}\binom{D-1}{\nu-1} \leq {\pi(\nu, D)} \leq \frac{\left(D + \frac{\nu(\nu-1)}{2}\right)^{\nu-1}}{\nu!(\nu-1)!}.
\end{equation}
The upper bound can be found in \cite{1972-upperbound}.
The lower bound is obvious upon noticing that every partition of $D$ can be obtained at most $\nu!$ times placing bars: $D = [1 + 1] + [1] + [1 + 1 + 1 + 1] \cdots$. Both the lower and upper bounds can be viewed as polynomials of degree $\nu - 1$ with the single indeterminate $D$ and coefficients dependent on $\nu$. Therefore, 
\begin{equation}
    \forall \nu \in \mathbb{N}: \quad \pi(\nu, D) \sim \frac{D^{\nu - 1}}{\nu!(\nu - 1)!}, \quad \text{ as } D \to \infty.
\end{equation}

        \subsection{\texorpdfstring{$\mathbb{T}$}{T} - Invariance.}
            In an angular case only directional components of relative atom positions are considered. In the plane they are described using only complex exponents $\phi_{m^t}(\theta)  = e^{im^t\theta}$. Therefore, to satisfy rotational invariance $\sum_{t=0}^{\nu} m^t = 0$ for a $\nu$-correlation basis function $\Phi_{\bbm} = \prod_{t = 1}^{\nu}\phi_{m^t}$. But as we will notice later, $m$-channel coupling is the dominant contribution to the asymptotic behavior of the portion of independent nodes in the $SO(2)$ and $O(3)$ cases as well. 

We will consider \textit{slices} of $\mathcal{K}_{\mathbb{T}}$ with fixed $\|\bk\|_1$. For $\mu, D \in \mathbb{N}$ we define an $\mathcal{E}$-\textit{slice} as
    
\begin{gather*}
    \mathcal{E}_{\mathbb{T}}(\mu, D) := \Big\{\bk \in \mathcal{K}_{\mathbb{T}}(\nu, D) \,\Big|\,  \|\bk\|_1 = D \Big\}  \\  
    = \Big\{\bbm \in \mathbb{Z}^{\mu}\,\Big|\,  m^1 \geq \dotsc \geq m^{\mu}, \quad \sum_{i = 1}^{\mu} m^i = 0, \quad \sum_{i = 1}^{\mu} |m^i| = D, \quad m^i \neq 0\Big\}.
\end{gather*}

We observe some straightforward properties:
\begin{enumerate}
    \item if $\bbm = \{m^i\} \in \mathcal{K}(\nu, D)$ is such that  $ m^1, \dotsc,  m^k \geq 0$ and $m^{k + 1}, \dotsc, m^{\nu} \leq 0$ then $\sum_{i = 1}^{k} m^i = -\sum_{i = k + 1}^{\nu} m^i$;
    \item if $D$ is odd then $\forall \nu$ $\mathcal{E}_{\mathbb{T}}(\nu, D) = \emptyset$ as the previous property cannot be satisfied;
    \item $\text{\#}\mathcal{K}_{\mathbb{T}}(\nu, D) - \text{\#}\mathcal{K}_{\mathbb{T}}(\nu, D - 1) = \sum_{k = 1}^{\nu} \text{\#}\mathcal{E}_{\mathbb{T}}(k, D)$, where $\mathcal{E}_{\mathbb{T}}(k, D)$ represents tuples with $\nu - k$ zero elements and $k$ non-zero.
\end{enumerate}

\begin{lemma} For $\nu \geq 3$ the asymptotic behaviour of an $\mathcal{E}$-slice is given by
    $$ \text{\#}\mathcal{E}_{\mathbb{T}}(\nu, D) \sim \frac{D^{\nu - 2}}{2^{\nu - 2}[(\nu - 1)!]^2}\binom{2\nu - 2}{\nu - 2} \qquad \text{as } D \to \infty.$$
\end{lemma}

\begin{proof}
    We will exploit property (1) and count the number of positive and negative element combinations separately. The index $k$ in the sum below indicates that a tuple $\{m^t\}_{t = 1}^{\nu}$ is considered with $m^1, \dotsc,  m^k > 0$ and $m^{k + 1}, \dotsc, m^{\nu} < 0$: 
    \begin{align}
        \text{\#}\mathcal{E}_{\mathbb{T}}(\nu, D) & = \sum_{k = 0}^{\nu}\pi\left(k, \frac{D}{2}\right)\pi\left(\nu - k, \frac{D}{2}\right)  \nonumber\\
        & \sim \sum_{k=0}^{\nu}\frac{\left(\frac{D}{2}\right)^{k-1}}{k!(k-1)!} \frac{\left(\frac{D}{2}\right)^{\nu - k-1}}{(\nu - k)!(\nu - k-1)!}  \nonumber\\
        & = \frac{D^{\nu - 2}}{2^{\nu - 2}(\nu!)^2}\sum_{k = 0}^{\nu}\frac{\nu!}{k!(\nu - k)!}\frac{\nu!\cdot k(\nu - k)}{k!(\nu - k)!}  \nonumber\\
        & = \frac{D^{\nu - 2}}{2^{\nu - 2}(\nu!)^2}\sum_{k = 0}^{\nu}\binom{\nu}{k}^2k(\nu - k) \nonumber\\
        & = \frac{D^{\nu - 2}}{2^{\nu - 2}(\nu!)^2}\nu^2\binom{2\nu - 2}{\nu - 2} = \frac{D^{\nu - 2}}{2^{\nu - 2}[(\nu - 1)!]^2}\binom{2\nu - 2}{\nu - 2}. \nonumber
        \qedhere 
    \end{align}
\end{proof}

Now, we can sum up slices to obtain the next lemma.

\begin{lemma} For $\nu, D \in \mathbb{N}$ the asymptotic behaviour of $\text{\#}\mathcal{K}_{\mathbb{T}}(\nu, D)$ is given by
    $$\text{\#}\mathcal{K}_{\mathbb{T}}(\nu, D) \sim \frac{D^{\nu - 1}}{2(\nu - 1) (2^{\nu - 2}[(\nu - 1)!]^2)}\binom{2\nu - 2}{\nu - 2} \qquad \text{as } D \to \infty.$$
    \label{lemma-T}
\end{lemma}
\begin{proof}
    Notice that a slice $\mathcal{E}(\nu, D)$ does not include tuples  $\bbm \ni \{0\}$, hence, to obtain all tuples that contain exactly $j$ zeros we should consider $\mathcal{E}(\nu - j, D)$. Therefore, 
    \begin{align}
        \text{\#}\mathcal{K}_{\mathbb{T}}(\nu, D) 
        & = 1 + \sum_{h = 1}^{D}\sum_{k = 1}^{\nu}\text{\#}\mathcal{E}(k, h)  \nonumber \\
        & \sim \sum_{h = 0}^{D}\delta(h\text{ is even})\sum_{k = 0}^{\nu}\frac{h^{k - 2}}{2^{k - 2}[(k - 1)!]^2}\binom{2k - 2}{k - 2} \nonumber\\
        & \sim \sum_{h = 0}^{D}\delta(h\text{ is even})\frac{h^{\nu - 2}}{2^{\nu - 2}[(\nu - 1)!]^2}\binom{2\nu - 2}{\nu - 2} \nonumber\\
        & \sim \frac{D^{\nu - 1}}{2(\nu - 1) (2^{\nu - 2}[(\nu - 1)!]^2)}\binom{2\nu - 2}{\nu - 2}. \nonumber
    \end{align}
    here $\delta(h\text{ is even}) = 1$ if $h = 2k$ and $\delta(h\text{ is even}) = 0$ if $h = 2k + 1$. 
\end{proof}





The next theorem states that dependent nodes constitute a vanishing minority of all nodes in the regime $D \gg \nu$. 

\begin{theorem} 
    For $\nu \geq 3$, 
    $$\frac{\text{\#}\mathcal{D}_{\mathbb{T}}(\nu, D)}{\text{\#}\mathcal{K}_{\mathbb{T}}(\nu, D)} = \Theta\left(\frac{1}{D}\right) \qquad \text{as } D \to \infty. $$
\end{theorem}
\begin{proof}
    The lower bound becomes obvious upon noticing that $\{ [0, \bbm] | \bbm \in \mathcal{K}_{\mathbb{T}}(\nu - 1, D) \} \subset \mathcal{D}_{\mathbb{T}}(\nu, D)$. 
    Since 
    \begin{equation}
        \text{\#}\mathcal{K}_{\mathbb{T}}(\nu - 1, D) \leq \text{\#}\mathcal{D}_{\mathbb{T}}(\nu, D) = \text{\#}\mathcal{K}_{\mathbb{T}}(\nu - 1, D) + \text{\#}\mathcal{D}^{\neq 0}_{\mathbb{T}}(\nu, D), 
    \end{equation}
    where $\mathcal{D}^{\neq 0}_{\mathbb{T}}(\nu, D) \subset \text{\#}\mathcal{D}_{\mathbb{T}}(\nu, D)$ is the set of dependent nodes that do not contain zero components, i.e. $\forall \bbm \in \mathcal{D}^{\neq 0}_{\mathbb{T}}(\nu, D): \, m^t \neq 0$.   The following inequality states that every dependent tuple can be split into two tuples of lower correlation order and degree. Equality is not satisfied due to the double counting caused by possible several separate decompositions of certain tuples    \Big(e.g. $[3, 2, 1, -1, -2, -3] = \big[[3, -3], \, [2, 1, -1, -2]\big] = \big[[2, -2], \,[3, 1, -1, -3]\big]$\Big):
    \begin{align*}
        \text{\#}\mathcal{D}^{\neq 0}_{\mathbb{T}}(\nu, D)
        &\leq \sum_{M = 0}^{D}\sum_{h = 0}^{M}\sum_{k = 0}^{\nu} 
            \text{\#}\mathcal{E}(k, h)\times \text{\#}\mathcal{E}(\nu - k, M - h)  \\ 
        &= O(D^2 (D^{k - 2}D^{\nu - k - 2}))  \\ 
        &= O(D^{\nu - 2}). \qedhere 
    \end{align*}
\end{proof}

        \subsection{\texorpdfstring{$O(3)$}{O(3)} and \texorpdfstring{$O(3)$}{O(2)} - Invariance.}
            We will only give proofs for the $O(3)$ invariant case, since the corresponding steps are completely analogous for the $SO(2)$ case. Recall that in the three-dimensional setting, the one-particle basis is given by $$\phi_{mln}(\br) = R_{n}(r)Y_{l}^{m}(\hat\br),$$
where $n \in \mathbb{N}, l \in \mathbb{N}$ and $m \in \mathbb{Z}$ with $|m| \leq l$. 
%

In the definition of $\mathcal{K}_{O(3)}$ we consider multisets of triplets $\big[(m^i, l^i, n^i)\big]_{i = 1}^{\nu}$ that can be written down as lexicographically ordered tuples of triplets. But we approach them from another perspective as triples of tuples $(\bbm, \bl, \bn)$. Note that with fixed $\bbm$ and $\bl$ some different permutations of $n$ can produce different elements of $\mathcal{K}_{O(3)}$. However, we estimate the number of tuples that satisfy all the above mentioned constraints but neglecting relative ordering of $\bbm$, $\bl$ and $\bn$, then if we mark the corresponding values with hats:

\begin{align*}
    \widehat{\mathcal{D}}_{O(3)}(\nu, D) &:= \Big\{(\bbm, \bl, \bn) \in  \big(\mathbb{Z}^{\nu}\big)^{3} \,\Big|\, m^1\geq \dotsc \geq m^{\nu},\, {\textstyle \sum_{t = 1}^{\nu} m^t = 0}, \quad {\textstyle \sum_{t=1}^{\nu}l^t} \text{ is even}, \nonumber\\
     & \hspace{7.7cm} {\textstyle |m^t| \leq l^t, \, \sum_{t = 1}^\nu n^t + l^t \leq D} \Big\}. \nonumber \\ 
     \widehat{\mathcal{K}}_{O(3)}(\nu, D) &:= \Big\{\bk \in  \widehat{\mathcal{K}}_{O(3)}(\nu, D)\,\Big|\, \bk \, \text{ is dependent} \Big\}. \nonumber
\end{align*}

Then, we have the bounds

\begin{align}
    \text{\#}\widehat{\mathcal{K}}_{O(3)}(\nu, D) &\leq  \text{\#}\mathcal{K}_{O(3)}(\nu, D) \leq (\nu!)^2\text{\#}\widehat{\mathcal{K}}_{O(3)}(\nu, D), \quad \text{and} \\
    \text{\#}\widehat{\mathcal{D}}_{O(3)}(\nu, D) &\leq  \text{\#}\mathcal{D}_{O(3)}(\nu, D) \leq (\nu!)^2\text{\#}\widehat{\mathcal{D}}_{O(3)}(\nu, D).
\end{align}

Next, we need the following technical lemma.
\begin{lemma} 
    For $\alpha, \beta \in \mathbb{N}$ we have 
    $$
         \sum_{k = 0}^{n}k^{\alpha}(n-k)^{\beta} \sim n^{\alpha + \beta + 1}\frac{\alpha!\beta!}{(\alpha + \beta + 1)!}
         \qquad \text{as } n \to \infty.
    $$
\end{lemma}
\begin{proof}
    The result is a straightforward application of a Riemann sum convering to the associated integral, 
    \begin{align}
        \frac{\sum_{k = 0}^{n}k^{\alpha}(n-k)^{\beta}}{n^{\alpha + \beta + 1}} & = \sum_{k = 0}^{n}\left(\frac{k}{n}\right)^{\alpha} \left(1- \frac{k}{n}\right)^{\beta}\frac{1}{n} \xrightarrow[n\to \infty]{} \int_0^1x^{\alpha}(1-x)^{\beta}dx\nonumber\\
        &  = \frac{\Gamma(\alpha + 1)\Gamma(\beta + 1)}{\Gamma(\alpha + \beta + 2)} =
        \frac{\alpha!\beta!}{(\alpha + \beta + 1)!}. \nonumber \qedhere 
    \end{align}
\end{proof}

We can now establish the prevalance of independent nodes in the $O(3)$ and $SO(2)$ cases.

\begin{theorem}
    If $\nu \geq 3$, then
    $$
        \frac{\text{\#}\mathcal{D}_{SO(2)}(\nu, D)}{\text{\#}\mathcal{K}_{SO(2)}(\nu, D)} = \Theta\left(\frac{1}{D}\right) \qquad \text{and} 
        \qquad
        \frac{\text{\#}\mathcal{D}_{O(3)}(\nu, D)}{\text{\#}\mathcal{K}_{O(3)}(\nu, D)} = \Theta\left(\frac{1}{D}\right) \qquad \text{as } D \to \infty.
    $$

    \label{theory:th:3dasymp}
\end{theorem}
\begin{proof}
    Suppose that $\deg(\bk) = \sum_{i = 1}^{\nu} l^i + \sum_{i = 1}^{\nu} n^i = H \leq D$, so let $L = \sum_{i = 1}^{\nu} l^i$ then $H - L = \sum_{i = 1}^{\nu} n^i$. Conceptually, 
    \begin{align}
        \label{eq:th:3dasymp:10}
        \text{\#}\widehat{\mathcal{K}}_{O(3)}(\nu, D) = \sum_{H=0}^{D}\sum_{L = 0}^{H}\text{\#}\Big\{\bbm : \sum |m^t| \leq L\Big\}\times\text{\#}\Big\{\bl\text{ that majorate } \bbm \text{ and } \sum l^t = L\Big\} \\ \times\text{\#}\Big\{\bn: \sum n^t = H - L\Big\}. \nonumber
    \end{align}
    Also suppose that we have $\{\bbm \in \mathbb{Z}^{\nu} : \sum m^t = 0, \, \sum |m^t| = M\}$ then to calculate the number of $\{\bl \in \mathbb{Z_+}^{\nu}: \sum l^t = L, \, l^t \geq |m^t|\}$, for every $\bbm$ we need to \textit{distribute} $L - M$ units over $\nu$ places as every $l^t$ is at least $|m^t|$. So if $\Lambda(\nu)$ and $\Lambda^{Dep}(\nu)$ are the corresponding asymptotic coefficients of the $\mathcal{E}$-slices of all and dependent nodes for the $\mathbb{T}$ case (i.e. \#$\mathcal{E}_{\mathbb{T}}(\nu, D) \sim \Lambda(\nu)\cdot D^{\nu-2}$). Then continuing on from \eqref{eq:th:3dasymp:10} we therefore get the $D \to \infty$ asymptotics
    
    \begin{align}
        \text{\#}\widehat{\mathcal{K}}_{O(3)}(\nu, D) & \sim \sum_{H = 0}^{D}\sum_{L = 0}^{H}\frac{1}{2}\left[\sum_{M=0}^{L}\left[\Lambda M^{\nu - 2}\times \frac{1}{2}\binom{L - M - 1}{\nu - 1}\right]\right]\frac{(H-L)^{\nu - 1}}{(\nu - 1)!\nu!}\nonumber\\
        & \sim \frac{\Lambda}{4(\nu-1)!^2\nu!}\sum_{H = 0}^{D}\sum_{L = 0}^{H}\left[\sum_{M=0}^{L}\left[M^{\nu - 2}(L-M)^{\nu - 1}\right]\right](H-L)^{\nu - 1}\nonumber\\ 
        & \sim \frac{\Lambda}{4(\nu-1)!^2\nu!}\frac{(\nu - 2)!(\nu-1)!}{(2\nu-2)!}\sum_{H = 0}^{D}\sum_{L = 0}^{H}L^{2\nu - 2}(H-L)^{\nu - 1}\nonumber\\
        & \sim \frac{\Lambda}{4(\nu-1)\nu!(2\nu-2)!}\frac{(2\nu-2)!(\nu - 1)!}{(3\nu - 2)!}\sum_{H = 0}^{D}H^{3\nu - 2}\nonumber\\
        & \sim \frac{\Lambda}{4(\nu-1)\nu(3\nu - 1)!}D^{3\nu - 1}\nonumber\\
        & = \frac{(2\nu - 2)!}{2^{\nu}(\nu-1)!^2\nu!^2(3\nu - 1)!}D^{3\nu - 1}. \nonumber
    \end{align}
    Similarly, 
    \begin{align}
        \text{\#}\widehat{\mathcal{D}}_{O(3)}(\nu, D) & = \sum_{H=0}^{D}\sum_{L = 0}^{H}\left[\text{\#}\Big\{\text{dependent } m\text{ - vectors}\Big\}\text{\#}\Big\{l\text{ - vectors that majorate}\Big\}\right]\pi(H - L, \nu)\nonumber \\
        & \sim \frac{\Lambda^{Dep}}{2(\nu-1)!^2\nu!}\sum_{H = 0}^{D}\sum_{L = 0}^{H}\frac{1}{2}\left[\sum_{M=0}^{L}\left[M^{\nu - 3}(L-M)^{\nu - 1}\right]\right](H-L)^{\nu - 1} \nonumber \\ 
        & \sim \frac{\Lambda^{Dep}}{4(\nu-2)(\nu-1)\nu(3\nu - 2)!}D^{3\nu - 2} = \frac{\Lambda^{Dep}}{\Lambda}\frac{(3\nu - 1)}{(\nu - 2)}\frac{1}{D} \text{\#}\widehat{\mathcal{K}}_{O(3)}(\nu, D). \nonumber
    \end{align}
    It is possible to use this scheme to  obtain $\widehat{\mathcal{K}}_{SO(2)} = \Theta(D^{2\nu-1})$ and $\widehat{\mathcal{D}}_{SO(2)} = \Theta(D^{2\nu-2})$.
\end{proof}

\subsection{Invariant features}

\label{sec:invariant_features}

We conclude our analysis with a brief remark on the case when particles are annotated with additional invariant features, such as chemical species, atomic mass, electric charge. For simplicity assume we have only one such additional feature, denoted by $\mu$, then the one-particle basis might take the form
\begin{equation}
    \phi_{nlmk}(\br, \mu) := R_n(r)Y_l^m(\hat\br)T_k(\mu),
\end{equation}
where $T_k$ is an additional polynomial basis with $k$ denoting the degree of $T_k$. The total degree of $\phi_{nlmk}$ is now defined as
\begin{equation}
    \deg(\phi_{nlmk}) = l + n + k.
\end{equation}
It turns out that in such a case the $1/D$ asymptotic ratio of the number of dependent to the number of total nodes is preserved:
\begin{theorem}
Denoting the corresponding sets by $\mathcal{K}^{f}_{O(3)}(\nu, D)$ and $\mathcal{D}^{f}_{O(3)}(\nu, D)$, we see that
    \[
        \frac{\text{\#}\mathcal{D}^{f}_{O(3)}(\nu, D)}{\text{\#}\mathcal{K}^{f}_{O(3)}(\nu, D)} = \Theta\left(\frac{1}{D}\right).
    \]
\end{theorem}
\begin{proof} 
It is possible estimate the number of nodes similarly as in in the theorem \ref{theory:th:3dasymp}:
    \begin{equation}
    \begin{aligned}
        \text{\#}\widehat{\mathcal{K}}^{f}_{O(3)}(\nu, D) = \sum_{H=0}^{D}\sum_{L = 0}^{H}\Big\{\bbm : \sum |m^t| \leq L\Big\}\Big\{\bl > \bbm : \sum l^t = L\Big\} \\
        \times\sum_{N=0}^{H-L}\Big\{\bn: \sum n^t = N\Big\}\Big\{\bk: \sum k^t = H - L - N\Big\} = \Theta(D^{4\nu - 1}),
    \end{aligned}
    \label{eq:invfeatures}
    \end{equation}
    if we additionally demand in (\ref{eq:invfeatures}) from nodes to be dependent, then additional constraints are applied to $\bbm$ and $\bl$  we see that $\text{\#}\widehat{\mathcal{D}}^{f}_{O(3)}(\nu, D) = \Theta(D^{4\nu - 1})$.
    \qedhere
\end{proof}  

\subsection{Complexity Analysis of Insertion Schemes}


\begin{proof}[Proof of Theorem \ref{theorem-alg1}: Complexity of Algorithm \ref{alg:original}]
Here $M$ = $\sum |m^j|$, $L = \sum l^j - |m^j|$, $N = \sum n^j$. Then, we have the bound 
    \begin{equation*}\begin{gathered}
        \#\mathcal{A}_{O(3)} \leq \sum_{\nu = 2}^{\nu_{max}}\Big\{\text{All possible tuples of order $\nu-1$}\Big\}\\ 
        = O\Bigg( ((\nu - 1)!)^2\sum_{M = \nu -1}^{D}2^{\nu - 1}\pi(M, \nu - 1)\sum_{L = 0}^{D - M}\pi(L, \nu - 1)\sum_{N = 0}^{D-M-N}\pi(N, \nu - 1)\Bigg)\\
        = O\Bigg(\Big(\sum_{n = 0}^{D}\pi(n, \nu - 1)\Big)^{3}\Bigg) = O\Big(\pi(D, \nu - 1)^3  D^3\Big) = O\Big(D^{3\nu - 3}\Big).
    \end{gathered}\end{equation*}
    We have already shown that $\mathcal{K}_{O(3)}(\nu, D) = \Theta\big(D^{3\nu - 1}\big)$, so 
    \begin{equation}
        \#\mathcal{V}_{O(3)}(\nu, D) = \Theta\Big( \sum_{k = 1}^{\nu}\#\mathcal{K}_{O(3)}(k, D) \Big) = \Theta\big(D^{3\nu - 1}\big).
    \end{equation}
    Analogously in the $SO(2)$ case we obtain $\#\mathcal{A}_{SO(2)} = O\big(D^{2\nu-2}\big)$ and $\mathcal{V}_{SO(2)}(\nu, D) = \Theta\big(D^{2\nu - 1}\big)$.
    
    The $\mathbb{T}$ case can be established by induction on $\nu$. First, suppose that $\nu = 3$, then we know that every independent node requires an auxiliary node to be inserted. Suppose that $\bbm = [m^1, -m^2, -m^3]$, where $m^t > 0$, then $m^1 = m^2 + m^3$, so $m^1$ has the largest absolute value, hence, the split by the Algorithm \ref{alg:original} will be $\bbm = {[m^1], [-m^2, -m^3]}$. Now we can conclude that all possible pairs $[m^1, m^2]$ and $[-m^1, -m^2]$ will be inserted. Therefore, 
    \begin{equation}
        \#\mathcal{A}_{\mathbb{T}}(\nu = 3, D) = 2\sum_{n = 1}^{\lfloor D/2 \rfloor}\pi(2, n).
    \end{equation}
    Now, suppose that $\nu \geq 4$, and all possible tuples $[m^t]_{t = 1}^{k}$ and $[-m^t]_{t = 1}^{k}$, where $k \leq \nu - 2$, are already inserted. Then if a tuple $\bbm = [m^t]_{t = 1}^{\nu}$ has $2 \leq k \leq \nu - 2$ positive values, so $m^{i_1} > 0, \dotsc, m^{i_k} > 0$ and $m^{i_{k+1}} < 0, \dotsc, m^{i_{\nu}} < 0$, then $\bbm$ can be inserted without any additional nodes. However, if there is only one positive or only one negative element $m^t$, then its absolute value is the largest in this tuple as $m^t = \sum_{s \neq t} m^s$, therefore, Algorithm \ref{alg:original} will insert $\bbm = \big[[m^t], [m^s]_{s \neq t}\big]$ and an additional node $[m^s]_{s \neq t}$, so
    \begin{equation}
        \#\mathcal{A}_{\mathbb{T}}(\nu, D) = \#\mathcal{A}_{\mathbb{T}}(\nu - 1, D) + 2\sum_{n = 1}^{\lfloor D/2 \rfloor}\pi(\nu - 1, n) = 2\sum_{k = 2}^{\nu - 1}\sum_{n = 1}^{\lfloor D/2 \rfloor} \pi(k, n).
    \end{equation}
    Considering the fact that $\pi(\nu, D) \sim \frac{D^{\nu - 1}}{\nu!(\nu - 1)!}$, we can conclude that 
    \begin{equation}
        \#\mathcal{A}_{\mathbb{T}}(\nu, D) \sim 2 \sum_{n = 1}^{\lfloor D/2 \rfloor} \frac{n^{\nu - 2}}{(\nu-1)!(\nu - 2)!} \sim \frac{D^{\nu-1}}{2^{\nu-2}(\nu-1)!^2},
    \end{equation}
    and using Lemma \ref{lemma-T} we obtain the statement of the theorem.
    \qedhere
\end{proof}


\begin{proof}[Proof of Theorem \ref{theorem-alg2}: Complexity of Algorithm \ref{alg:modn}]
We employ an analogous argument as in the previous proof to asses the computational complexity of Algorithm \ref{alg:modn}. As above, we can estimate 
    \begin{equation}\begin{gathered}
        \#\mathcal{A}_{O(3)} \leq \sum_{\nu = 2}^{\nu_{max}}\Big\{\text{All possible tuples of order $\nu-n$}\Big\}\\ 
        = O\Bigg( ((\nu - n)!)^2\sum_{M = \nu -1}^{D}2^{\nu - n}\pi(M, \nu - 1)\sum_{L = 0}^{D - M}\pi(L, \nu - n)\sum_{N = 0}^{D-M-N}\pi(N, \nu - n)\Bigg)\\
        = O\Bigg(\Big(\sum_{t = 0}^{D}\pi(t, \nu - n)\Big)^{3}\Bigg) = O\Big(\pi(D, \nu - n)^3  D^3\Big) = O\Big(D^{3\nu - 3n}\Big).
    \end{gathered}\end{equation}
Similarly for $SO(2)$ and $\mathbb{T}$ we have 
    \begin{equation*}
        \#\mathcal{A}_{SO(2)} = O\Bigg(\Big(\sum_{t = 0}^{D}\pi(t, \nu - n)\Big)^{2}\Bigg) = O\Big(D^{2\nu - 2n}\Big); \qquad \#\mathcal{A}_{\mathbb{T}} = O\Big(D^{\nu - n}\Big).
        \qedhere
    \end{equation*}
\end{proof}
        
    \printbibliography
\end{document}